\newif\iffull
\fulltrue

\iffull
\documentclass[a4paper,10pt]{article}
\usepackage[utf8]{inputenc}
\usepackage{amsfonts, amssymb, amsmath, amsthm}
\usepackage{mathtools}
\usepackage{enumerate}
\usepackage{graphicx}
\usepackage{fullpage}
\usepackage{verbatim}
\usepackage{hyperref}
\usepackage{authblk}
\hypersetup{colorlinks=true, linkcolor=red, urlcolor=blue, citecolor=blue, pdfpagemode=UseNone, pdfstartview=FitH, bookmarksopen=true}
\hypersetup{pdftitle=Title}

\newtheorem*{theorem*}{Theorem}
\newtheorem{theorem}{Theorem}

\newtheorem*{lemma*}{Lemma}
\newtheorem{lemma}[theorem]{Lemma}

\newtheorem*{proposition*}{Proposition}
\newtheorem{proposition}[theorem]{Proposition}
\newtheorem*{fact*}{Fact}

\newtheorem*{question*}{Question}

\newtheorem{conjecture}[theorem]{Conjecture}
\newtheorem*{corollary*}{Corollary}

\numberwithin{claimcounter}{theorem}
\newtheorem*{claim*}{Claim}

\theoremstyle{remark}
\newtheorem*{remark*}{Remark}
\newtheorem{remark}[theorem]{Remark}

\theoremstyle{definition}
\newtheorem*{definition*}{Definition}

\newtheorem*{observation*}{Observation}

\title{Optimal bounds for the colorful fractional Helly theorem\thanks{D.~B. is
supported by GA\v{C}R grant no. 19-27871X. A.~G. is supported by KAW-stipendiet
2015.0360 from the Knut and Alice Wallenberg Foundation. M.~T. is supported by the GA\v{C}R grant 19-04113Y.}}
\author[1]{Denys Bulavka} 
\author[2]{Afshin Goodarzi} 
\author[1]{Martin Tancer}
\affil[1]{\small Department of Applied Mathematics, Charles University,
Malostransk\'{e} n\'{a}m.
25, 118~00~~Praha~1, Czech Republic}
\affil[2]{\small Royal Institute of Technology, Department of Mathematics,
S-100~44, Stockholm, Sweden}
\date{}

\else
\documentclass[a4paper,UKenglish,cleveref, autoref]{socg-lipics-v2019}
\title{Optimal bounds for the colorful fractional Helly theorem}

\author{Denys Bulavka}{Department of Applied Mathematics, Charles University,
Malostransk\'{e} n\'{a}m.
25, 118~00~~Praha~1, Czech
Republic}{dbulavka@kam.mff.cuni.cz}{0000-0002-4119-9402}{Supported by the GA\v{C}R grant 19-27871X.}
\author{Afshin Goodarzi}{Royal Institute of Technology, Department of
Mathematics,
S-100~44, Stockholm, Sweden}{afshingo@kth.se}{}{Supported by KAW-stipendiet
2015.0360 from the Knut and Alice Wallenberg Foundation.}
\author{Martin Tancer}{Department of Applied Mathematics, Charles University,
Malostransk\'{e} n\'{a}m.
25, 118~00~~Praha~1, Czech Republic}{`my surname'@kam.mff.cuni.cz}{}{Supported
by the GA\v{C}R grant 19-04113Y.}

\authorrunning{D. Bulavka, A. Goodarzi, M. Tancer}
\Copyright{Denys Bulavka, Afshin Goodarzi, Martin Tancer}



\acknowledgements{We thank Xavier Goaoc and Eran Nevo for providing us with useful references.}

\theoremstyle{theorem}
\newtheorem{conjecture}[theorem]{Conjecture}

\fi

\newcommand{\R}{\mathbb{R}}

\newcommand{\N}{\mathbb{N}}
\newcommand{\Q}{\mathbb{Q}}
\newcommand{\F}{\mathcal{F}}

\newcommand{\KK}{\mathsf{K}}
\newcommand{\LL}{\mathsf{L}}

\newcommand{\PP}{\mathbf{P}}

\DeclareMathOperator{\SPAN}{span}
\DeclareMathOperator{\col}{col}

\newcommand{\avec}{{\mathbf a}}
\newcommand{\bvec}{{\mathbf b}}
\newcommand{\rvec}{{\mathbf r}}
\newcommand{\kvec}{{\mathbf k}}
\newcommand{\nvec}{{\mathbf n}}
\newcommand{\unitvec}{{\mathbf 1}}

\newcommand{\lvec}{\boldsymbol{\ell}}
\newcommand{\tvec}{{\mathbf t}}

\newcommand{\lip}{{\llcorner}}
\newcommand{\bbeta}{\boldsymbol{\beta}}
\newcommand{\eepsilon}{\boldsymbol{\varepsilon}}

\newif\ifcmnts
\cmntsfalse 
\ifcmnts
\newcommand{\marrow}{\marginpar{\boldmath$\longleftarrow$}}
\newcommand{\martin}[1]{\ifhmode\newline\fi\marrow
  \textsf{\textcolor{green}{\bf
MARTIN:} #1\newline}}
\newcommand{\afshin}[1]{\ifhmode\newline\fi\marrow
  \textsf{\textcolor{green}{\bf
      Afshin:} #1\newline}}

\newcommand{\denys}[1]{\ifhmode\newline\fi\marrow
  \textsf{\textcolor{green}{\bf
      Denys:} #1\newline}}
\else
\newcommand{\martin}[1]{}
\newcommand{\afshin}[1]{}
\newcommand{\denys}[1]{}
\fi

\newcommand{\blue}[1]{\begingroup\color{blue}#1\endgroup}

\begin{document}

\maketitle

\begin{abstract}
The well known fractional Helly theorem and colorful Helly theorem can
be merged into the so called colorful fractional Helly theorem. It states: for every
$\alpha \in (0, 1]$ and every non-negative integer $d$, there is $\beta_{\col}
  = \beta_{\col}(\alpha, d)
\in (0, 1]$ with the following property. Let $\mathcal{F}_1, \dots, \mathcal{F}_{d+1}$ be finite nonempty
families of convex sets in $\mathbb{R}^d$ of sizes $n_1, \dots, n_{d+1}$ respectively. If at
least $\alpha n_1 n_2 \cdots n_{d+1}$ of the colorful $(d+1)$-tuples have a nonempty
intersection, then there is $i \in [d+1]$ such that $\mathcal{F}_i$ contains a subfamily of
size at least $\beta_{\col} n_i$ with a nonempty intersection. (A colorful
  $(d+1)$-tuple is a $(d+1)$-tuple $(F_1, \dots ,
F_{d+1})$ such that $F_i$ belongs to $\mathcal{F}_i$ for every $i$.)

The colorful fractional Helly theorem was first stated and proved
by Bárány, Fodor, Montejano, Oliveros, and Pór in 2014 with $\beta_{\col} = \alpha/(d+1)$.
In 2017 Kim proved the theorem with better function $\beta_{\col}$, which in particular
tends to $1$ when $\alpha$ tends to $1$. Kim also conjectured what is the optimal bound
for $\beta_{\col}(\alpha, d)$ and provided the upper bound example for the optimal bound.
The conjectured bound coincides with the optimal bounds for the (non-colorful)
fractional Helly theorem proved independently by Eckhoff and Kalai around 1984.

We verify Kim's conjecture by extending Kalai's approach to the colorful scenario.
Moreover, we obtain optimal bounds also in more general setting when we allow
  several sets of the same color. 
\end{abstract}

\section{Introduction}
 The target of this paper is to provide optimal bounds for the colorful fractional
 Helly theorem first stated by B\'{a}r\'{a}ny, Fodor, Montejano, Oliveros,
 and~P\'{o}r~\cite{barany-fodor-montejano-oliveros-por14}, and then improved by
 Kim~\cite{kim17}. In order to explain the colorful fractional
  Helly theorem, let us briefly survey the preceding results.


The starting point, as usual in this context, is the Helly
theorem:

\begin{theorem}[Helly's theorem~\cite{helly23}]
  \label{t:helly}
  Let $\F$ be a finite family of at least $d+1$ convex sets in $\R^d$. Assume that every
  subfamily of $\F$ with exactly $d+1$ members has a nonempty intersection.
  Then all sets in $\F$ have a nonempty intersection.
\end{theorem}

Helly's theorem admits numerous extensions and two of them, important in our
context, are the fractional Helly theorem and the colorful Helly theorem. The
fractional Helly theorem of Katchalski and Liu covers the case when only some fraction of the $d+1$
tuples in $\F$ has a nonempty intersection.

\begin{theorem}[The fractional Helly theorem~\cite{katchalski-liu79}]
  \label{t:fh}
  For every $\alpha \in (0, 1]$ and every non-negative integer $d$, there is
  $\beta = \beta(\alpha, d) \in (0, 1]$ with the following property. Let 
  $\F$ be a finite family of $n \geq d+1$ convex sets in $\R^d$ such that at least
  $\alpha \binom{n}{d+1}$ of the subfamilies of $\F$ with exactly $d+1$ members
  have a nonempty intersection. Then there is a subfamily of $\F$ with at least
  $\beta n$ members with a nonempty intersection.
\end{theorem}

An interesting aspect of the fractional Helly theorem is not only to show the
existence of $\beta(\alpha, d)$ but also to provide the largest value of
$\beta(\alpha, d)$ with which the theorem is valid. This has been resolved
independently by Eckhoff~\cite{eckhoff85} and by Kalai~\cite{kalai84}  showing
that the fractional Helly theorem holds with $\beta(\alpha, d) = 1 -
(1-\alpha)^{1/(d+1)}$; yet another simplified proof of this fact has been
subsequently given by Alon and Kalai~\cite{alon-kalai85}. It is well known that this bound is sharp by
considering a family $\F$ consisting of $\approx (1 -
(1-\alpha)^{1/(d+1)})n$ copies of $\R^d$ and $\approx (1-\alpha)^{1/(d+1)}n$
hyperplanes in general position; see, e.g., the introduction of~\cite{kalai84}.

The colorful Helly theorem of Lov\'{a}sz covers the case where the sets are
colored by $d+1$ colors and only the `colorful' $(d+1)$-tuples of sets in $\F$ are
considered. Given families $\F_1, \dots, \F_{d+1}$ of sets in $\R^d$ a family
of sets $\{F_1, \dots, F_{d+1}\}$ is a \emph{colorful $(d+1)$-tuple} if $F_i
\in \F_i$ for $i \in [d+1]$, where $[n] :=\{1,\dots ,n\}$ for a non-negative integer $n\geq 1$. (The reader may think of $\F$ from preceding
theorems decomposed into color classes $\F_1, \dots, \F_{d+1}$.)

\begin{theorem}[The colorful Helly theorem~\cite{lovasz74, barany82}]
  \label{t:ch}
  Let $\F_1, \dots, \F_{d+1}$ be finite nonempty families of convex sets in
  $\R^d$. Let us assume that every colorful $(d+1)$-tuple has a nonempty
  intersection.
  Then
  one of the families $\F_1, \dots, \F_{d+1}$ has a nonempty intersection.
\end{theorem}

Both the colorful Helly theorem and the fractional Helly theorem with optimal
bounds imply the Helly theorem. The colorful one by setting $\F_1 = \cdots =
\F_{d+1} = \F$ and the fractional one by setting $\alpha = 1$ giving $\beta(1,
d) = 1$.

The preceding two theorems can be merged into the following colorful
fractional Helly theorem:

\begin{theorem}[The colorful fractional Helly
  theorem~\cite{barany-fodor-montejano-oliveros-por14}]
\label{t:cfh}
  For every $\alpha \in (0, 1]$ and every non-negative integer $d$, there is
  $\beta_{\col} = \beta_{\col}(\alpha, d) \in (0, 1]$ with the following property.
  Let $\F_1, \dots, \F_{d+1}$ be finite nonempty families of convex sets in
  $\R^d$ of sizes $n_1, \dots, n_{d+1}$ respectively. If at least $\alpha n_1
  \cdots n_{d+1}$ of the colorful $(d+1)$-tuples have a nonempty intersection,
  then there is $i \in [d+1]$ such that $\F_i$ contains a subfamily of size at
  least $\beta_{\col} n_i$ with a nonempty intersection. 
\end{theorem}

B\'{a}r\'{a}ny et al. proved the colorful fractional Helly theorem with the
value $\beta_{\col}(\alpha, d) = \frac{\alpha}{d+1}$ and they used it
as a lemma~\cite[Lemma~3]{barany-fodor-montejano-oliveros-por14} in a proof of
a colorful variant of a $(p,q)$-theorem. Despite this, the optimal bound for
$\beta_{\col}$ seems to be of independent interest. In particular, the bound on
$\beta_{\col}$ has been subsequently improved by Kim~\cite{kim17} who showed
that the colorful fractional Helly theorem is true with $\beta_{\col}(\alpha,
d) =
\max\{\frac{\alpha}{d+1}, 1 - (d+1)(1-\alpha)^{1/(d+1)}\}$. On the other hand,
the value of $\beta_{\col}(\alpha, d)$ cannot go beyond $1 - (1 -
\alpha)^{1/(d+1)}$ because essentially the same example as for the standard
fractional Helly theorem applies in this setting as well---it is
sufficient to set $n_1 = n_2 = \cdots = n_{d+1}$ and take $\approx (1 -
(1-\alpha)^{1/(d+1)})n_i$ copies of $\R^d$ and $\approx (1-\alpha)^{1/(d+1)}n_i$
hyperplanes in general position in each color class.\footnote{At the end of
Section~\ref{s:mixed_colorful} we discuss this example in full detail in more
general context. However, in this special case, it is perhaps much easier to
check directly that $\beta_{\col}$ cannot be improved due to this example.} (Kim~\cite{kim17}
provides a slightly different upper bound example showing the same bound.)

Coming back to the lower bound on $\beta_{\col}(\alpha,d)$, Kim explicitly
conjectured that $1 - (1 - \alpha)^{1/(d+1)}$ is also a lower bound, thereby an
optimal bound for the colorful fractional Helly theorem. He also provides a
more refined conjecture, that we discuss slightly later on (see
Conjecture~\ref{c:kim_refined}), which implies
this lower bound. We prove the refined conjecture, and therefore
the optimal bounds for the colorful fractional Helly theorem.

\begin{theorem}[The optimal colorful fractional Helly theorem]
\label{t:cfh_optimal_convex}
  Let $\F_1, \dots, \F_{d+1}$ be finite nonempty families of convex sets in
  $\R^d$ of sizes $n_1, \dots, n_{d+1}$ respectively. If at least $\alpha n_1
  \cdots n_{d+1}$ of the colorful $(d+1)$-tuples have a nonempty intersection,
  for $\alpha \in (0, 1]$, 
  then there is $i \in [d+1]$ such that $\F_i$ contains a subfamily of size at
  least $(1 - (1-\alpha)^{1/(d+1)}) n_i$ with a nonempty intersection. 
\end{theorem}

In the proof we follow the exterior algebra approach which has been used by
Kalai~\cite{kalai84} in order to provide optimal bounds for the standard
fractional Helly theorem. We have to upgrade Kalai's proof to the colorful
setting. This requires guessing the right generalization of several steps in
Kalai's proof (in particular guessing the statement of Theorem~\ref{t:c-k-p}
below). However, we honestly admit that after making these `guesses' we
follow Kalai's proof quite straightforwardly.


Let us also compare one aspect of our proof with the previous proof of the
weaker bound by Kim~\cite{kim17}: Kim's proof uses the colorful Helly theorem
as a blackbox while our proof includes the proof of the colorful Helly theorem.

Last but not least, the exterior algebra approach actually allows to generalize
Theorem~\ref{t:cfh_optimal_convex} in several different directions. The
extension to so called $d$-collapsible complexes is essentially mandatory for
the well working proof while the other generalizations that we will present
just follow from the method. We will discuss this in detail in forthcoming
subsections of the introduction.

\subsection{$d$-representable and $d$-collapsible complexes}

\iffull\paragraph{The nerve and $d$-representable complexes.} 
\fi
The important information
in Theorems~\ref{t:helly},~\ref{t:fh},~\ref{t:ch},~\ref{t:cfh},
and~\ref{t:cfh_optimal_convex} is which subfamiles have a nonempty intersection.
This information can be efficiently stored in a simplicial complex called the nerve.

A \emph{(finite abstract) simplicial complex} is a set system $\KK$ on a finite
\emph{set of vertices} $N$ such that whenever $A \in \KK$ and $B \subseteq A$, then
$B \in \KK$. (The standard notation for the vertex set would be $V$ but
this notation will be more useful later on when we will often use capital
letters such as $R$ for some set and the corresponding lower case letters such
as $r$ for its size.)
The elements of $\KK$ are \emph{faces} (a.k.a. \emph{simplices})
of $\KK$. The \emph{dimension} of a face $A \in \KK$ is defined as $\dim A =
|A|-1$; this corresponds to representing $A$ as an
$(|A|-1)$-dimensional simplex. The dimension of $\KK$, denoted $\dim \KK$, is
the maximum of the dimensions of faces in $\KK$.
A face of dimension $k$ is a
\emph{$k$-face} in short. \emph{Vertices} of $\KK$ are usually identified with
$0$-faces, that is, $v \in N$ is identified with $\{v\} \in \KK$. (Though the definition
of simplicial complex allows that $\{v\} \not\in \KK$ for $v
\in N$, in our applications we will always have $\{v\} \in \KK$ for $v \in N$.) 
Given a family of sets
$\F$, the \emph{nerve} of $\F$, 
is the simplicial
complex whose vertex set is $\F$ and whose faces are subfamilies with a
nonempty intersection. A simplicial complex is \emph{$d$-representable} if it
is the nerve of a finite family of convex sets in~$\R^d$. 

As a preparation for the $d$-collapsible setting, we now restate
Theorem~\ref{t:cfh_optimal_convex} in terms of $d$-representable complexes. For
this we need two more notions. Given a simplicial complex $\KK$ and a subset $U$
of the vertex set $N$, the \emph{induced subcomplex} $\KK[U]$ is defined as $\KK[U]
:= \{A \in \KK\colon A \subseteq U\}$. Now, let us assume that the
vertex set $N$ is split into $d+1$ pairwise disjoint subsets $N = N_1 \sqcup
\cdots \sqcup N_{d+1}$ (we can think of this partition as coloring each vertex
of $N$ with one of the $d+1$ possible colors). Then a \emph{colorful $d$-face} is a
$d$-face $A$, such that $|A \cap N_i| = 1$ for every $i \in [d+1]$.

\begin{theorem}[Theorem~\ref{t:cfh_optimal_convex} reformulated]
\label{t:cfh_optimal_representable}
 Let $\KK$ be a $d$-representable simplicial complex with the set of vertices
  $N = N_1
  \sqcup \cdots \sqcup N_{d+1}$ divided into $d+1$ disjoint subsets. Let $n_i
  := |N_i|$ for $i \in [d+1]$ and assume that $\KK$ contains
  at least $\alpha n_1 \cdots n_{d+1}$ colorful $d$-faces for some $\alpha \in
  (0, 1]$. Then there is $i \in [d+1]$ such that $\dim \KK[N_i] \geq (1 -
  (1-\alpha)^{1/(d+1)}) n_i - 1$. 
\end{theorem}

 Theorem~\ref{t:cfh_optimal_representable} is indeed just a reformulation of
Theorem~\ref{t:cfh_optimal_convex}: Considering $\F$ as disjoint
union\footnote{If there are any repetitions of sets in $\F$, which we generally
allow for families of sets, then each repetition creates a new vertex in
the nerve.} $\F = \F_1 \sqcup \cdots \sqcup \F_{d+1}$, then $\KK$ corresponds to the
nerve of $\F$, colorful $d$-faces correspond to colorful $(d+1)$-tuples with
nonempty intersection and the dimension of $\KK[V_i]$ corresponds to the size of
largest subfamily of $\F_i$ with nonempty intersection minus $1$. (The shift by
minus 1 between size of a face and dimension of a face is a bit unpleasant; however, we want to follow the standard terminology.)

\iffull
\paragraph{$d$-collapsible complexes.}
\fi

In~\cite{wegner75} Wegner introduced an important class of simplicial
complexes, called $d$-collapsible complexes. They include all
$d$-representable complexes, which is the main result of~\cite{wegner75},
while they admit quite simple combinatorial description which is useful for
induction.

Given a simplicial complex $\KK$, we say that a simplicial complex $\KK'$ arises
from $\KK$ by an \emph{elementary $d$-collapse}, if there are faces $L, M \in
\KK$ with
the following properties: (i) $\dim L \leq d-1$; (ii) $M$ is the unique
inclusion-wise maximal face which contains $L$; and (iii) $\KK' = \KK \setminus \{A
\in \KK \colon L \subseteq A\}$. A simplicial complex $\KK$ is
\emph{$d$-collapsible} if there is a sequence of simplicial complexes $\KK_0,
\dots, \KK_{\ell}$ such that $\KK = \KK_0$; $\KK_i$ arises from $\KK_{i-1}$ by an
elementary $d$-collapse for $i \in [\ell]$; and $\KK_{\ell}$ is the empty
complex.


We will prove the following generalization of
Theorem~\ref{t:cfh_optimal_representable} (equivalently of
Theorem~\ref{t:cfh_optimal_convex}).

\begin{theorem}[The optimal colorful fractional Helly theorem for
  $d$-collapsible complexes]
\label{t:cfh_optimal_collapsible}
 Let $\KK$ be a $d$-collapsible simplicial complex with the set of vertices $N
  = N_1
  \sqcup \cdots \sqcup N_{d+1}$ divided into $d+1$ disjoint subsets. Let $n_i
  := |N_i|$ for $i \in [d+1]$ and assume that $\KK$ contains
  at least $\alpha n_1 \cdots n_{d+1}$ colorful $d$-faces for some $\alpha \in
  (0, 1]$. Then there is $i \in [d+1]$ such that $\dim \KK[N_i] \geq (1 -
  (1-\alpha)^{1/(d+1)}) n_i - 1$. 
\end{theorem}

\subsection{Kim's refined conjecture and further generalization}
As a tool for a possible proof of Theorem~\ref{t:cfh_optimal_convex},
Kim~\cite[Conjecture~4.2]{kim17} suggested the following conjecture. (The
notation $k_i$ in Kim's statement of the conjecture is our $r_i+1$.)

\begin{conjecture}[\cite{kim17}]
\label{c:kim_refined}
  Let $n_i$ be positive and $r_i$ non-negative integers for $i \in [d+1]$ with $n_i \geq
  r_i + 1$. Let $\F_1, \dots, \F_{d+1}$ be families of convex sets in $\R^d$ such
  that $|\F_i| = n_i$ and there is no subfamily of $\F_i$ of size $r_i + 1$ with
  non-empty intersection for every $i \in [d+1]$. Then the number of colorful $(d+1)$-tuples
  with nonempty intersection is at most 
\[
  n_1 \cdots n_{d+1} - (n_1 - r_1) \cdots(n_{d+1} - r_{d+1}).
\]
\end{conjecture}

We explicitly prove this conjecture in a slightly more general setting for
$d$-collapsible complexes. (Note that the condition `no subfamily of 
size $r_i + 1$' translates as `no $r_i$-face', that is, `the dimension is at
most $r_i - 1$'.)

\begin{proposition}
\label{p:kim_collapsible}
  Let $n_i$ be positive and $r_i$ non-negative integers for $i \in [d+1]$ with $n_i \geq
  r_i + 1$.
 Let $\KK$ be a $d$-collapsible simplicial complex with the set of vertices $N =
  N_1 \sqcup \cdots \sqcup N_{d+1}$ divided into $d+1$ disjoint subsets. Assume
  that $|N_i| = n_i$ and that $\dim \KK[N_i] \leq r_i - 1$ for every $i \in
  [d+1]$. Then $\KK$ contains at most 
\[
  n_1 \cdots n_{d+1} - (n_1 - r_1) \cdots(n_{d+1} - r_{d+1}).
\]
 colorful $d$-faces.  
\end{proposition}

\iffull \paragraph{Our main technical result.} 
\fi
Now, let us present our main technical
tool for a proof of Proposition~\ref{p:kim_collapsible} and consequently for a proof
of Theorem~\ref{t:cfh_optimal_collapsible} as well. 

We denote by $\N$ the set of positive integers whereas $\N_0$ is
the set of non-negative integers. Let us consider $c \in \N$ and
vectors $\kvec = (k_1,
\dots, k_c), \rvec = (r_1,
\dots, r_c) \in \N_0^c$ and $\nvec = (n_1, \dots, n_c) \in \N^c$ such that $\kvec, \rvec \leq \nvec$. (Here the
notation $\avec \leq \bvec$ means that $\avec$ is less or equal to $\bvec$ in
every coordinate.) We will also use the
notation $k := k_1 + \cdots + k_c$,
$n := n_1 + \cdots + n_c$, and $r := r_1 + \cdots + r_c$.
Let $N$ be a set with $n$ elements partitioned as $N = N_1 \sqcup \cdots
\sqcup N_c$ where $|N_i| = n_i$ for $i \in [c]$. By $\binom{N}{\kvec}$ we
denote the set of all subsets $A$ of $N$ such that $|A \cap N_i| = k_i$ for
every $i \in [c]$. Note that $\binom{N}{\kvec} \subseteq \binom{N}{k}$ where
$\binom{N}{k}$ denotes the set of all subsets of $N$ of size $k$.

Let $\KK$ be a simplicial complex with the vertex set $N$ as above. We say
that a face $A$ of $\KK$ is $\kvec$-colorful if $A \in \binom{N}{\kvec}$, that
is, $|A \cap N_i| = k_i$ for every $i \in [c]$. The earlier notion of colorful face
corresponds to setting $c= d+1$ and $\kvec = \unitvec := (1,\dots, 1)\in \N^c$. By $f_{\kvec} = f_{\kvec}(\KK)$ we denote the
\emph{$\kvec$-colorful $f$-vector} of $\KK$, that is, the number of $\kvec$-colorful
faces in $\KK$.

Let us further assume that we are given sets $R_i \subseteq  N_i$ with $|R_i| =
r_i$ for every $i \in [c]$. Let $R = R_1 \sqcup \cdots \sqcup R_{c}$ and
$\bar R := N \setminus R$. Then, we define the set system

\begin{equation*}
P_{\kvec}(\nvec,d,\rvec) = \left \{S \in \binom{N}{\kvec}: |S\cap \bar{R}| \leq d \right \}.
\end{equation*}

We remark that $P_{\kvec}(\nvec,d,\rvec)$ is not a simplicial complex, as it
contains only sets in $\binom{N}{\kvec}$. However, this set system is useful
for estimating the number of $\kvec$-colorful faces in a $d$-collapsible
complex. 
%
By $p_{\kvec}(\nvec,d,\rvec)$ we denote the size of $P_{\kvec}(\nvec,d,\rvec)$,
that is, $p_{\kvec}(\nvec,d,\rvec) :=
|P_{\kvec}(\nvec,d,\rvec)|$.


\begin{theorem}
  \label{t:c-k-p}
  For integers $c,d\geq 1$, let $\KK$ be a $d$-collapsible simplicial complex
  with vertex partition $N = N_1\sqcup \dots \sqcup N_{c}$ and let $\nvec =
  (n_1, \dots, n_c)
  \in \N^{c}$ be the vector with $n_i=|N_i|$. For $\rvec=(r_1, \dots,
  r_c) \in \N^c$ such
  that $\dim \KK[N_i] \leq r_i - 1$ for $i\in[c]$ and $\kvec \in \N_0^{c}$
  such that $\kvec \leq \nvec$ it follows that 
  \begin{equation*}
    f_{\kvec}(\KK) \leq p_{\kvec}(\nvec,d,\rvec).
  \end{equation*}
\end{theorem}


Theorem~\ref{t:c-k-p} is proved in Section~\ref{s:exterior}. 
Here we show the implications Theorem~\ref{t:c-k-p} $\Rightarrow$
Proposition~\ref{p:kim_collapsible} and Proposition~\ref{p:kim_collapsible}
$\Rightarrow$ Theorem~\ref{t:cfh_optimal_collapsible}. In addition, we
advertise that Theorem~\ref{t:c-k-p} yields further generalizations of
Theorem~\ref{t:cfh_optimal_collapsible}. We explain this last part in
Section~\ref{s:mixed_colorful}.

\begin{proof}[Proof of Proposition~\ref{p:kim_collapsible} modulo
  Theorem~\ref{t:c-k-p}] 
We use Theorem~\ref{t:c-k-p} with $c = d+1$ and $\kvec = \unitvec$. Then it is
  sufficient to compute $p_{\unitvec}(\nvec, d, \rvec)$. On the one hand, the size of
  $\binom{N}{\unitvec}$ is $n_1 \dots n_{d+1}$. On the other hand, $A$ belongs to
  $\binom{N}{\unitvec} \setminus P_{\unitvec}(\nvec, d, \rvec)$ if and only if
  $|A \cap (N_i \setminus R_i)| = 1$ for every $i \in [d+1]$. Then, the
  number of such $A$ is $(n_1 - r_1)\cdots(n_{d+1}-r_{d+1})$. Combining
  these observations we obtain the required formula
 \[
   p_{\unitvec}(\nvec, d, \rvec) = n_1 \dots n_{d+1} - (n_1 -
   r_1)\cdots(n_{d+1}-r_{d+1}).
 \]
\end{proof}

\begin{proof}[Proof of Theorem~\ref{t:cfh_optimal_collapsible} modulo
  Proposition~\ref{p:kim_collapsible}]
  By contradiction, let us assume that for every $i \in [d+1]$ we get $\dim
  \KK[N_i] < (1-(1-\alpha)^{1/(d+1)})n_i - 1$. Let us set $r_i := \dim \KK[N_i] +
  1 < (1-(1-\alpha)^{1/(d+1)})n_i$. Then Proposition~\ref{p:kim_collapsible}
  gives that the number of colorful $d$-faces is at most
\begin{equation*}
\prod_{i=1}^{d+1} n_i - \prod_{i=1}^{d+1}(n_i-r_i) < \prod_{i=1}^{d+1}n_i - (1-(1-(1-\alpha)^{1/(d+1)}))^{d+1}\prod_{i=1}^{d+1}n_i = \alpha \prod_{i=1}^{d+1}n_i
\end{equation*}
which is a contradiction due to the strict inequality on the first line.
\end{proof}

\section{Exterior algebra}
In this section we prove Theorem~\ref{t:c-k-p}.
First we overview the required tools from exterior algebra---here we follow~\cite[Section 2]{kalai84} very closely.

\label{s:exterior}
Let $N$ be a finite set of $n$ elements (with a fixed total order $\leq$)
and let $V = \R^N$ be the
$n$-dimensional real vector space with standard basis vectors $e_i$ for $i \in N$. Let $\bigwedge V$ be the $2^n$ dimensional exterior algebra over $V$ with basis
vectors $e_S$ for $S \subseteq N$. The exterior product $\wedge$ on this
algebra is defined so that it satisfies (i) $e_\emptyset$ is a neutral element,
that is $e_\emptyset \wedge e_S = e_S = e_S \wedge e_\emptyset$; (ii) $e_S = e_{i_1} \wedge \cdots
\wedge e_{i_s}$ for $S = \{i_1, \dots, i_s\} \subseteq N$ where $i_1 < \cdots <
i_s$ and we identify $e_i$ with $e_{\{i\}}$ for $i \in N$ (iii) $e_i \wedge e_j =
-e_j \wedge e_i$ for $i,j \in N$. By $\bigwedge^k V$ we denote the subspace of
$\bigwedge V$ generated by $(e_S)_{S \in \binom Nk}$ where $0 \leq k \leq n$.
We consider the standard inner product on
both $V$ and $\bigwedge V$ so that $(e_i)_{i \in N}$ and $(e_S)_{S
\subseteq N}$ are their orthonormal bases respectively. Then $(e_S)_{S \in
\binom Nk}$ is also an orthonormal basis of $\bigwedge^k V$. 

Given another basis $(g_i)_{i \in N}$, let $A = (a_{ij})_{i,j \in N}$ be the $N
\times N$ transition matrix\footnote{Here we index rows and columns of a matrix
by elements from some set, not necessarily integers. That is by $N \times N$ matrix we mean the matrix where both rows and
columns are indexed by elements of $N$.} from $(e_i)_{i \in N}$ to $(g_i)_{i \in N}$, that
is, $g_i = \sum_{j\in N} a_{ij} e_j$ for any $i \in N$. The basis $(g_i)_{i \in
N}$ induces a basis of $\bigwedge V$ given by $g_S = g_{i_1} \wedge \cdots
\wedge g_{i_s}$ for $S = \{i_1, \dots, i_s\} \subseteq N$. Transition from the
standard basis $(e_S)_{S \in \binom{N}{k}}$ of $\bigwedge^k V$ to $(g_S)_{S \in
\binom{N}{k}}$ is given by
\begin{equation}
g_S = \sum_{T \in \binom{N}{k}} \det A_{S|T} e_T
  \label{e:gs_standard}
\end{equation}
where $A_{S|T} = (a_{ij})_{i \in S, j \in T}$ for $S, T \subseteq N$.

Given an $m$-element set $M$ and $M \times N$-matrix $A$ and $k \leq m,n$, let
$C_k(A)$ be the \emph{compound} matrix $(\det A_{S|T})_{S \in \binom{M}{k}, T \in \binom{N}{k}}$.

The following lemma is implicitly contained in~\cite{kalai84}.

\begin{lemma}
\label{l:columns}
If the columns of $A$ are linearly independent, then the columns of $C_k(A)$ are linearly independent as well.
\end{lemma}
\begin{proof}
If columns of $A$ are linearly independent, then $n \leq m$. Consider an
  arbitrary square submatrix $B$ of rank $n$. Considering $B$ as a transition
  matrix from $(e_i)_{i \in N}$ to $(g_i)_{i \in N}$, we get that $C_k(B)$ is a transition matrix from  $(e_S)_{S \in \binom{N}{k}}$ to $(g_S)_{S \in
\binom{N}{k}}$, thus $C_k(B)$ has full rank. However, $C_k(B)$ is also a submatrix of $C_k(A)$ with all $\binom{n}{k}$ columns.
\end{proof}

Now, let us in addition assume that $(g_i)_{i \in N}$ is an orthonormal basis of $V$. As
pointed out by Kalai, it follows from the Cauchy-Binet formula that $(g_S)_{S
\in \binom Nk}$ is also an orthonormal basis of $\bigwedge^k V$.

For $f,g\in \bigwedge V$ we define its \emph{left interior product}, denoted by
$g \lip f$, as the unique element in $\bigwedge V$
which satisfies that $\langle u, g \lip f \rangle = \langle u \wedge g, f
\rangle$ for all $u\in \bigwedge V$. It turns out that $g_T \lip g_S$
is non-zero only if $T \subseteq S$, in which case $g_T \lip g_S = \pm g_{S\setminus T}$.

\iffull \paragraph{Colored exterior algebra.}
\fi
Now we extend the previous tools to the colored setting.
From now on, let us assume that $N$ is an $n$-element set decomposed into
$c$-color classes, $N = N_1 \sqcup \cdots \sqcup N_{c}$. (The total order on
$N$ in this case starts with elements of $N_1$, then continues with elements of
$N_2$, etc.)
We pick an $N
\times N$-matrix $A$ so that it is a block-diagonal matrix with blocks
corresponding to  individual $N_i$. That is, $A_{N_i|N_j}$ is a zero matrix
whenever $i \neq j$. On the other hand, as shown by Kalai~\cite[Section~2]{kalai84}, it is possible to
pick each $A_{N_i|N_i}$ so that $(g_j)_{j \in N_i}$ is an orthonormal basis of
the subspace of $V$ generated by $(e_j)_{j \in N_i}$ and each square
submatrix of $A_{N_i|N_i}$ has full rank. Therefore, from now on, we assume
that we picked $A$ and the vectors $g_j$ this way. (Such a block matrix, for
$c=2$, is previously mentioned in~\cite{nevo05}.)

Similarly as in the introduction, let us set $\nvec = (n_1, \dots, n_c)$ so that $n_i = |N_i|$ for $i
\in [c]$; for simplicity, let us assume that each $N_i$ is nonempty---that
is, $\nvec$ is a $c$-tuple of positive integers. Let us also
consider another $c$-tuple $\kvec = (k_1, \dots, k_c)$ of non-negative
integers such that $\kvec \leq \nvec$ and we set $k := k_1 + \cdots + k_c$.
Then by $\bigwedge^{\kvec} V$ we mean the subspace of $\bigwedge V$
generated by $(e_S)_{S \in \binom N\kvec}$; recall that $\binom N\kvec$ is the
set of all subsets $A$ of $N$ such that $|A \cap N_i| = k_i$ and that $\binom
  N\kvec \subseteq \binom Nk$. Thus we also get that
$\bigwedge^{\kvec} V$ is a subspace of $\bigwedge^k V$. In addition, due to our
choice of $(g_j)_{j \in N}$ we get that $g_S \in \bigwedge^{\kvec} V$ if $S \in
\binom N\kvec$. In addition $\det A_{S|T} = 0$ if $T \in \binom Nk
\setminus \binom N\kvec$ because $A_{S|T}$ is in this case a block matrix such
that some of the blocks is not a square. Thus the formula~\eqref{e:gs_standard} simplifies to
\begin{equation}
g_S = \sum_{T \in \binom{N}{\kvec}} \det A_{S|T} e_T.
  \label{e:gs_colorful}
\end{equation}

\begin{proof}[Proof of Theorem~\ref{t:c-k-p}]

For $\kvec \in \N^{c}$ such that $k\leq d$ we have that
  $P_{\kvec}(\nvec, d,\rvec) = \binom{N}{\kvec}$, thus the theorem follows
trivially. On the other hand, if $k > r$, then $k_i > r_i$ for some $i$ and
consequently $f_{\kvec}(\KK) = 0$ due to our assumption $\dim \KK[N_i] \leq r_i - 1$;
  therefore the theorem again follows trivially. From now on we assume $d+1
  \leq k
  \leq r$. (We also use the notation for the sets $R$, $\bar R$ and $R_i$ with
  $|R_i| = r_i$ as in
  the definition of $P_{\kvec}(\nvec, d,\rvec)$.)

Let us define the following subspaces of
  $\bigwedge^{\kvec} V$
\begin{equation*}
A_{\kvec} := \left \{m\in \bigwedge^{\kvec}V :\left (\forall T\in \binom{R}{k - d} \right ) \text{ }g_T \lip m = 0 \right \},
\end{equation*}
and
\begin{equation*}
W_{\kvec} := \SPAN \left \{e_S \in \bigwedge^{\kvec} V: S \in
  \binom{N}{\kvec}\text{ and } S \in \KK \right \},
\end{equation*}
from the definition it follows that the colorful $f$-vector and the dimension of $W_{\kvec}$ coincide, i.e.  $f_{\kvec} = \dim(W_{\kvec})$.

We claim that
\begin{equation*}
  \dim(A_{\kvec}) \geq \left |\binom{N}{\kvec}\right | - p_{\kvec}(\nvec,d,\rvec).
\end{equation*}
Indeed, if $S\in \binom{N}{\kvec}$ such that $S \notin P_{\kvec}(\nvec,d,\rvec)$, then
  $|S\cap \bar{R}| > d$. As $S \subseteq R\sqcup \bar{R} = N$ and $|S| = k$ we have
  that $|S \cap R| < k - d$. If $T\in \binom{R}{k-d}$ we have that
  $S\nsupseteq T$; therefore $g_T\lip g_S = 0$.
  From this it follows that $g_S \in A_{\kvec}$ and finally the claim because
  $g_S\in \bigwedge^{\kvec} V$.

  The core of the proof is to show $A_{\kvec} \cap W_{\kvec} = \{ 0 \}$. Once
  we have this, we get $f_k(\KK) = \dim(W_k) \leq \dim \bigwedge^{\kvec} V - \dim
  A_{\kvec} \leq p_{\kvec}(\nvec,d,\rvec)$ which proves the theorem.

  For contradiction, let $m\in A_k\cap
  W_k$ be a non-zero element. Because $m \in W_{\kvec}$, it can be
  written as $m = \sum \alpha_Se_S$ where the sum is over all $S \in \binom
  N\kvec$ such that $S \in \KK$. Let $\KK_0, \dots, \KK_{\ell}$ be a sequence of
  simplicial complexes showing $d$-collapsibility of $\KK$ as in the definition
  of $d$-collapsible complex. In addition, due to~\cite[Lemma~3.2]{kalai84}, it
  is possible to assume that $\KK_i$
  arises from $\KK_{i-1}$ by so called \emph{special} elementary $d$-collapse which
  is either a removal of a maximal face of dimension at most $d-1$ or the
  minimal face (the face $L$ in the definition) has dimension exactly $d-1$. 
%
%
Now let us consider the
  first step from $\KK_{i-1}$ to $\KK_i$ such that a face $U \in
  \binom{N}{\kvec}$ with non-zero $\alpha_U$ is eliminated. Denote by $L$ and $M$
  the faces determining the collapse as in the definition. We have $L \subseteq
  U \subseteq M$,
  $|M|\geq |U|=k>d$ and therefore $|L| = d$ (equivalently, $\dim L = d-1$), because
  the collapse is special. For $T\in
  \binom{R}{k - d}$ let $\tvec =(t_1, \dots,t_{c})\in \N^{c}$
  be such that $t_i = |T\cap N_i|$. Then $g_T =
    \sum_{P\in \binom{N}{\tvec}} \det (A_{T|P})e_P$ via~\eqref{e:gs_colorful}.
    We also need to simplify the expression $\langle e_L, g_T \lip e_S \rangle$
    for $S \in \binom{N}{\kvec}$.
We obtain
\begin{equation}
  \langle e_L, g_T \lip e_S \rangle = \langle e_L\wedge g_T, e_S \rangle =
  \sum_{P\in \binom{N}{\tvec}} \det (A_{T|P}) \langle e_L\wedge e_P,e_S
  \rangle
\end{equation}

  If $S\nsupseteq L$  then $\langle e_L\wedge e_P,e_S
    \rangle = 0$ for all $P$, and therefore $\langle e_L, g_T \lip e_S \rangle = 0$.
If $S\supseteq L$ then $\langle e_L \wedge e_P, e_S\rangle = 0$ unless $P = S\setminus L$ and therefore
$\langle e_L, g_T \lip e_S\rangle = \langle e_L \wedge e_{S\setminus L}, e_S \rangle  \det(A_{T|S\setminus L})$.

Since $m \in A_k$, for arbitrary $T\in \binom{R}{k-d}$ we get
\begin{align*}
  0 &= \langle e_L, g_T \lip m \rangle
    = \sum_{\substack{S\in \binom{N}{\kvec}:  S\in \KK}} \alpha_S \langle
    e_L,g_T \lip e_S \rangle 
    = \sum_{S\in \binom{N}{\kvec}: S\in \KK_{i-1}} \alpha_S\langle e_L,g_T \lip
    e_S\rangle\\
    &= \sum_{S\in \binom{N}{\kvec}:S \supseteq L} \alpha_S \langle
    e_L,g_T\lip e_S\rangle
  =\sum_{S\in \binom{N}{\kvec}:M \supseteq S \supseteq L} \alpha_S \langle e_L\wedge
  e_{S\setminus L}, e_S\rangle \det(A_{T|S\setminus L})
\end{align*}
where the third equality follows from the fact that $\alpha_S = 0$ for $S \in
\KK
\setminus \KK_{i-1}$ due to our choice of $\KK_{i-1}$ and the last two equalities
follow from our earlier simplification of $\langle e_L,g_T \lip  e_S\rangle$.
(We also use that the expressions $S \supseteq L$ and $M \supseteq S \supseteq
L$ are equivalent as $M$ is the unique maximal face containing $L$.)

We also have $U \in \binom{N}{\kvec}$ with $M \supseteq U \supseteq
L$ for which $\alpha_U \neq
0$ as well as $\langle e_L\wedge  e_{U\setminus L}, e_U\rangle$ is nonzero (the
latter one equals $\pm 1$). Therefore the expression above is a linear dependence of the columns of
$C_{k-d}(A_{R|M\setminus L})$. However, we will also show that the columns of
$C_{k-d}(A_{R|M\setminus L})$ are linearly independent, thereby getting a
contradiction. Via Lemma~\ref{l:columns}, it is sufficient to check that
the columns of $A_{R|M\setminus L}$ are linearly independent.  Because $A$ is a
block-matrix with blocks $A_{N_i|N_i}$, we get that $A_{R|M\setminus L}$ is a
block matrix with blocks $A_{R_i|(M\setminus L)\cap N_i}$. Thus it is
sufficient to check that the columns are independent in each block. But this
follows from our assumptions of how we picked $A$ in each block, using that $|R_i| =
r_i \geq |(M\setminus L)\cap N_i|$ as $|M \cap N_i| \leq r_i$ due to our
assumption $\dim \KK[V_i] \leq r_{i}-1$.
\end{proof}

\section{$\kvec$-colorful fractional Helly theorem}
\label{s:mixed_colorful}
Theorem~\ref{t:c-k-p} allows to generalize
Theorem~\ref{t:cfh_optimal_collapsible} in two more
directions.

The first generalization of
Theorem~\ref{t:cfh_optimal_collapsible} is already touched in the introduction.
We can deduce an analogy of Theorem~\ref{t:cfh_optimal_collapsible} for
$\kvec$-colorful faces (instead of just colorful $d$-faces) where $\kvec =
(k_1, \dots, k_c) \in \N_0^{c}$ is some vector with $c \geq 1$. For example, if $d = 2$,
$\kvec = (2, 1,1)$ and we understand the partition of $N = N_1 \sqcup N_2 \sqcup
N_3$ as coloring the vertices of $\KK$ red, green, or blue. Then we seek for
number of faces that contain two red vertices, one green vertex and one blue
vertex.

For the second generalization, let us first observe that in the conclusion of Theorem~\ref{t:cfh_optimal_collapsible} there is the
same coefficient $1 - (1- \alpha)^{1/(d+1)}$ independently of~$i$. However, in
the notation of Theorem~\ref{t:cfh_optimal_collapsible}, we 
may also seek for $i$ such that $\dim \KK[N_i] \geq \beta_i n_i + 1$ where
$\bbeta = (\beta_1, \dots, \beta_c) \in (0,1]^c$ is some fixed vector.
Then for given $\bbeta$, we want to find the lowest $\alpha \in (0,1]$ with
which we reach the conclusion analogous as in
Theorem~\ref{t:cfh_optimal_collapsible}.
This is a natural analogy of various Ramsey type statements: for example, if
the edges of a complete graph $G$ with at least $9$ vertices are colored blue or red,
then the graph contains either a blue copy of the complete graph on $3$ vertices
or a red copy of the complete graph on $4$ vertices.

For the purpose of stating the generalization, let us set 
\begin{equation}
\label{e:L_kd}
  L_\kvec(d) := \{\lvec = (\ell_1, \cdots \ell_c) \in \N_0^{c} \colon
  \ell_1 + \cdots + \ell_c \leq d \hbox{ and }\ell_i \leq
  k_i \hbox{ for }i \in[c]\}
\end{equation}
and
\begin{equation}
\label{e:alpha_kdbeta}
  \alpha_\kvec(d, \bbeta) :=  \sum_{\lvec = (\ell_1, \dots,
  \ell_{c})\in L_{\kvec}(d)}
  \prod_{i=1}^{c} \binom{k_i}{\ell_i}
  (1-\beta_i)^{\ell_i}(\beta_i)^{k_i-\ell_i}. 
\end{equation}

\begin{theorem}
\label{t:kcfh_optimal_collapsible}
Let $c, d \geq 1$ and $\kvec = (k_1, \dots, k_c) \in \N_0^{c}$ be such that
  $k := k_1 + \cdots + k_{c} \geq d+1$. Let $\KK$ be a $d$-collapsible simplicial complex with the set of vertices $N
  = N_1
  \sqcup \cdots \sqcup N_c$ divided into $c$ disjoint subsets. Let $n_i
  := |N_i|$ for $i \in [c]$ and assume that $\KK$ contains at least 
  $\alpha_{\kvec}(d, \bbeta) \big|\binom{N}{\kvec}\big|$ $\kvec$-colorful faces 
  for some $\bbeta = (\beta_1, \dots, \beta_c) \in (0, 1]^{c}$. Then
  there is $i \in [c]$ such that $\dim \KK[N_i] \geq  \beta_i n_i - 1$.
\end{theorem}

The formula~\eqref{e:alpha_kdbeta} for $\alpha_{\kvec}(d, \bbeta)$ in
Theorem~\ref{t:kcfh_optimal_collapsible} is, unfortunately, a bit complicated. However, this is
the optimal value for $\alpha$ in the theorem. We first prove
Theorem~\ref{t:kcfh_optimal_collapsible} and then we will provide an example
showing that for every $d$, $\kvec$ and $\bbeta$ as in the theorem, the value
for $\alpha$ cannot be improved. 
The remark below is a probabilistic interpretation of~\eqref{e:alpha_kdbeta}. (This, for
example, easily reveals that $\alpha_{\kvec}(d, \bbeta) \in (0, 1]$ for given
parameters and will help us with checking monotonicity in $\bbeta$.)

\begin{remark} 
\label{r:prob_alpha}
  Consider a random experiment where we gradually for each $i$ pick
  $k_i$ numbers $x^i_{1}, \dots, x^{i}_{k_i}$ in the interval $[0,1]$
  independently at random (with uniform distribution).
  Let $\ell_i$ be the number of $x^i_j$ which are greater than $\beta_i$ and let us
  consider the event $A_{\kvec}(d, \bbeta)$ expressing that $\ell_1 + \cdots
  + \ell_{c} \leq d$. Then $\alpha_{\kvec}(d, \bbeta)$ is the probability
  $\PP[A_{\kvec}(d, \bbeta)]$.
  

Indeed, the probability that the number of $x^i_j$ which are greater than
  $\beta_i$ is exactly $\ell_i$ is given by the expression beyond the sum
  in~\eqref{e:alpha_kdbeta}. 
%
  Therefore, we need to sum this over all options giving $\ell_1 + \cdots
    + \ell_{c} \leq d$ and $\ell_i \leq k_i$.
\end{remark}

In the proof of Theorem~\ref{t:kcfh_optimal_collapsible} we will need the
following slightly modified proposition. We relax `at least' to `more
than' while we aim at strict inequality in the conclusion---this innocent change will be a significant advantage in the proof. On
the other hand, after this change we can drop the assumption $k \geq d+1$. But
this is only a cosmetic change, because the proposition below is vacuous if
$\alpha_{\kvec}(d, \bbeta) = 1$ which in particular happens if $k < d+1$.

\begin{proposition}
\label{p:kcfh_optimal_collapsible}
  Let $c,d \geq 1$ and $\kvec = (k_1, \dots, k_{c}) \in \N_0^{c}$.
  Let $\KK$ be a $d$-collapsible simplicial complex with the set of vertices $N
  = N_1
  \sqcup \cdots \sqcup N_{c}$ divided into $c$ disjoint subsets. Let $n_i
  := |N_i|$ for $i \in [c]$ and assume that $\KK$ contains more than 
  $\alpha_{\kvec}(d, \bbeta) \big|\binom{N}{\kvec}\big|$ $\kvec$-colorful faces 
  for some $\bbeta = (\beta_1, \dots, \beta_{c}) \in (0, 1]^{c}$. Then there
  is $i \in [c]$ such that $\dim \KK[N_i] >  \beta_i n_i - 1$.
\end{proposition}

First we show how Theorem~\ref{t:kcfh_optimal_collapsible} follows from
Proposition~\ref{p:kcfh_optimal_collapsible} by a limit transition. Then we
prove Proposition~\ref{p:kcfh_optimal_collapsible}.

\begin{proof}[Proof of Theorem~\ref{t:kcfh_optimal_collapsible} modulo
  Proposition~\ref{p:kcfh_optimal_collapsible}.]
Let us consider $\varepsilon > 0$ such that $\bbeta - \eepsilon \in (0,
  1]^{c}$ for $\eepsilon = (\varepsilon, \dots, \varepsilon) \in
  (0,1]^{c}$. 
  
  First, we need to check $\alpha_\kvec(d, \bbeta) > \alpha_\kvec(d,
  \bbeta - \eepsilon)$. For this we will use Remark~\ref{r:prob_alpha} and we
  also use $k \geq d+1$. It is easy to check $A_{\kvec}(d, \bbeta) \supseteq A_{\kvec}(d,
  \bbeta - \eepsilon)$ which gives $\alpha_\kvec(d, \bbeta) \geq \alpha_\kvec(d,
    \bbeta - \eepsilon)$. In order to show the strict inequality, it remains to
    show that $A_{\kvec}(d, \bbeta) \setminus A_{\kvec}(d, \bbeta - \eepsilon)$ has
    positive probability. Consider the output of the experiment when each
    $x_i^j \in (\beta_i - \varepsilon, \beta_i)$. This output has positive probability
    $\varepsilon^k$. In addition, this output belongs to $A_{\kvec}(d, \bbeta)$
    whereas it does not belong to $A_{\kvec}(d, \bbeta - \eepsilon)$ (becuase
    $k \geq d+1$) as required.

 This means, that we can apply Proposition~\ref{p:kcfh_optimal_collapsible}
  with $\alpha_\kvec(d, \bbeta - \eepsilon)$ as we know that $\KK$ has at least
  $\alpha_{\kvec}(d, \bbeta) \big|\binom{N}{\kvec}\big|$ $\kvec$-colorful faces
  by assumptions of Theorem~\ref{t:kcfh_optimal_collapsible} which is more than $\alpha_{\kvec}(d, \bbeta - \varepsilon)
  \big|\binom{N}{\kvec}\big|$. We obtain $\dim \KK[N_i] > (\beta_i -
  \varepsilon) n_i - 1$. By letting $\varepsilon$ to tend to $0$, we obtain the
  required $\dim \KK[N_i] \geq \beta_i n_i - 1$.
\end{proof}

\iffull \paragraph{Boosting the complex.}
\fi
In the proof of Proposition~\ref{p:kcfh_optimal_collapsible}, we will need the
following procedure for boosting the complex. For a given complex $\KK$ with
vertex set $N = N_1 \sqcup \cdots \sqcup N_{c}$ partitioned as usual, and a
non-negative integer $m$ we define the complex 
$K_{\langle m \rangle}$ as a complex with the vertex set $N \times [m] = N_1 \times
[m] \sqcup \cdots \sqcup N_{c} \times [m]$ whose maximal faces are of the
form $S \times [m]$, where $S$ is a maximal face of $K$. We will also use
the notation $\delta_\kvec(\KK) := f_\kvec(K)/|\binom{N}{\kvec}|$ for the
density of $\kvec$-colorful faces of $\KK$. 

\begin{lemma}
\label{l:boosting}
Let $\KK$ be a simplicial complex with vertex partition $N=N_1\sqcup \dots
  \sqcup N_{c}$ and $\kvec = (k_1, \dots, k_{c}) \in \N_0^{c}$, then
  \begin{enumerate}[(i)]
\item
  $\delta_{\kvec}(\KK_{\langle m \rangle}) \geq \delta_{\kvec}(\KK)$; and 
    \item if $\KK$ is $d$-collapsible, then $\KK_{\langle m \rangle}$ is
      $d$-collapsible as well.
\end{enumerate}
\end{lemma}
\begin{proof}
  Let us start with the proof of (i).
  If $\delta_{\kvec}(\KK) = 0$ there is nothing to prove. Thus we may assume that
  $\delta_{\kvec}(\KK) > 0$ (equivalently $f_\kvec(\KK) > 0$) and consequently
  we have that $|N_i| \geq k_i$. Let us interpret $\delta_{\kvec}(\KK)$ as the
  probability that a random $\kvec$-tuple of vertices in $N$ is a simplex of
  $\KK$, and we interpret $\delta_{\kvec}(\KK_{\langle m \rangle})$
  analogously. Let $\pi\colon N \times [m] \to N$ be the projection to the first
  coordinate. 
  Now, let $U$ be a $\kvec$-tuple of vertices in $N\times [m]$ taken uniformly
  at random. Considering the set $\pi(U) \subseteq N$, it need not be a
  $\kvec$-tuple (this happens exactly when two points in $U$ have the same
  image under $\pi$) but it can be extended to a $\kvec$-tuple $W$ using that
  $|N_i| \geq k_i$ for every $i$. Let $W$ be an extension of $\pi(U)$
  to a $\kvec$-tuple, taken uniformly at random among all possible choices.
  Because of the choices we made, $W$ is in fact a $\kvec$-tuple of vertices in
  $N$ taken uniformly at random. (Note that the choices done in each $N_i$ or
  $N_i \times [m]$ are independent of each other.)
  Altogether, using $\PP$ for probability, we get
  \[ 
  \delta_{\kvec}(\KK_{\langle m \rangle}) = \PP[U \in \KK_{\langle m \rangle}] =
  \PP[\pi(U) \in \KK] \geq \PP[W \in \KK] = \delta_{\kvec}(\KK).
  \]
  This shows (i).

  For~(ii), we follow the idea of splitting a vertex
  from~\cite[Proposition~14(i)]{alon-kalai-matousek-meshulam02} which proves a
  similar statement for $d$-Leray complexes. For a complex $\KK$ and a vertex
  $v \in \KK$ let $\KK^{v \to v_1, v_2}$ be a complex obtained from $\KK$ by
  splitting the vertex $v$ into two newly introduced vertices $v_1$ and $v_2$.
  That is, if $V$ is the set of vertices of $\KK$, then the set of vertices of
  $\KK^{v \to v_1, v_2}$ is $(V\cup \{v_1,v_2\})\setminus\{v\}$ assuming $v_1,
  v_2 \not\in V$. The maximal simplices of $\KK^{v \to v_1, v_2}$ are obtained
  from maximal simplices $S$ of $\KK$ by replacing $v$ with $v_1$ and $v_2$, if
  $S$ contains $v$ (otherwise $S$ is kept as it is). Our aim is to show that if
  $\KK$ is $d$-collapsible, then $\KK^{v \to v_1, v_2}$ is $d$-collapsible as
  well. This will prove~(ii) because $\KK_{\langle m \rangle}$ can be obtained
  from $\KK$ by repeatedly splitting some vertex. For the proof, we extend the
  notation $\KK^{v \to v_1, v_2}$ by setting $\KK^{v \to v_1, v_2} = \KK$ if
  $v$ does not belong to $\KK$.

  Let $\KK_0 = \KK, \KK_1, \dots, \KK_\ell = \emptyset$ be a sequence such that
  $\KK_i$ arises from $\KK_{i-1}$ by an elementary $d$-collapse. Our task is to
  show that $\KK^{v \to v_1, v_2}_{i-1}$ $d$-collapses to $\KK_i^{v \to v_1, v_2}$
  for $i \in [\ell]$. This will show the claim as $\KK^{v \to v_1, v_2}_\ell =
  \emptyset$. For simplicity of the notation, we will treat only the elementary
  $d$-collapse from $\KK$ to $\KK_1$ as other steps are analogous. We will assume $v \in
  \KK$, as there is nothing to do if $v \not \in \KK$.

  Let $L$ and $M$ be the faces from the elementary
  $d$-collapse. That is, $\dim L \leq d-1$; $M$ is the unique maximal face in
  $\KK$ which contains $L$ and $\KK_1$ is obtained from $\KK$ by removing all
  faces that contain $L$, including $L$. We will distinguish three cases
  according to whether $v \in L$ or $v \in M$.

  If $v \not\in M$ (which implies $v \not\in L$), then $M$ is the unique
  maximal face containing $L$ in $\KK^{v \to v_1, v_2}$ and the elementary
  $d$-collapse removing $L$ and all its superfaces yields $\KK_1^{v \to v_1,
  v_2}$.

  If $v \in M$ while $v \not \in L$, then $(M \cup \{v_1, v_2\})\setminus
  \{v\}$ is the unique maximal face containing $L$ in $\KK^{v \to v_1, v_2}$
  and the elementary $d$-collapse removing $L$ and all its superfaces yields $\KK_1^{v \to v_1,
      v_2}$.

  Finally, if $v \in M$ and $v \in L$, then we need to perform the $d$-collapse
  from $\KK^{v \to v_1, v_2}$ to $\KK_1^{v \to v_1, v_2}$ by two elementary
  steps; see Figure~\ref{f:doubled_in_L}. First we realize that $(M \cup \{v_1, v_2\})\setminus
    \{v\}$ is the unique maximal face containing $(L \cup \{v_1\}) \setminus
    \{v\}$ in $\KK^{v \to v_1, v_2}$. Because $\dim (L \cup \{v_1\}) \setminus
    \{v\} = \dim L$, we can perform an elementary $d$-collapse removing $(L \cup
    \{v_1\}) \setminus \{v\}$ and all its superfaces obtaining a complex
    $\KK'$. In $\KK'$  we have that $(M \cup \{v_2\}) \setminus \{v\}$ is the
    unique maximal face containing $(L \cup \{v_2\}) \setminus \{v\}$. After
    removing $(L \cup \{v_2\}) \setminus \{v\}$ and all its superfaces, we get
    desired $\KK_1^{v \to v_1, v_2}$ (note that in this case $\KK_1^{v \to v_1,
    v_2}$ is indeed obtained from $\KK^{v \to v_1, v_2}$ by removing $(L \cup
        \{v_1\}) \setminus \{v\}$, $(L \cup \{v_2\}) \setminus \{v\}$ and all
	their superfaces).
\begin{figure}
  \begin{center}
    \includegraphics[width=\textwidth]{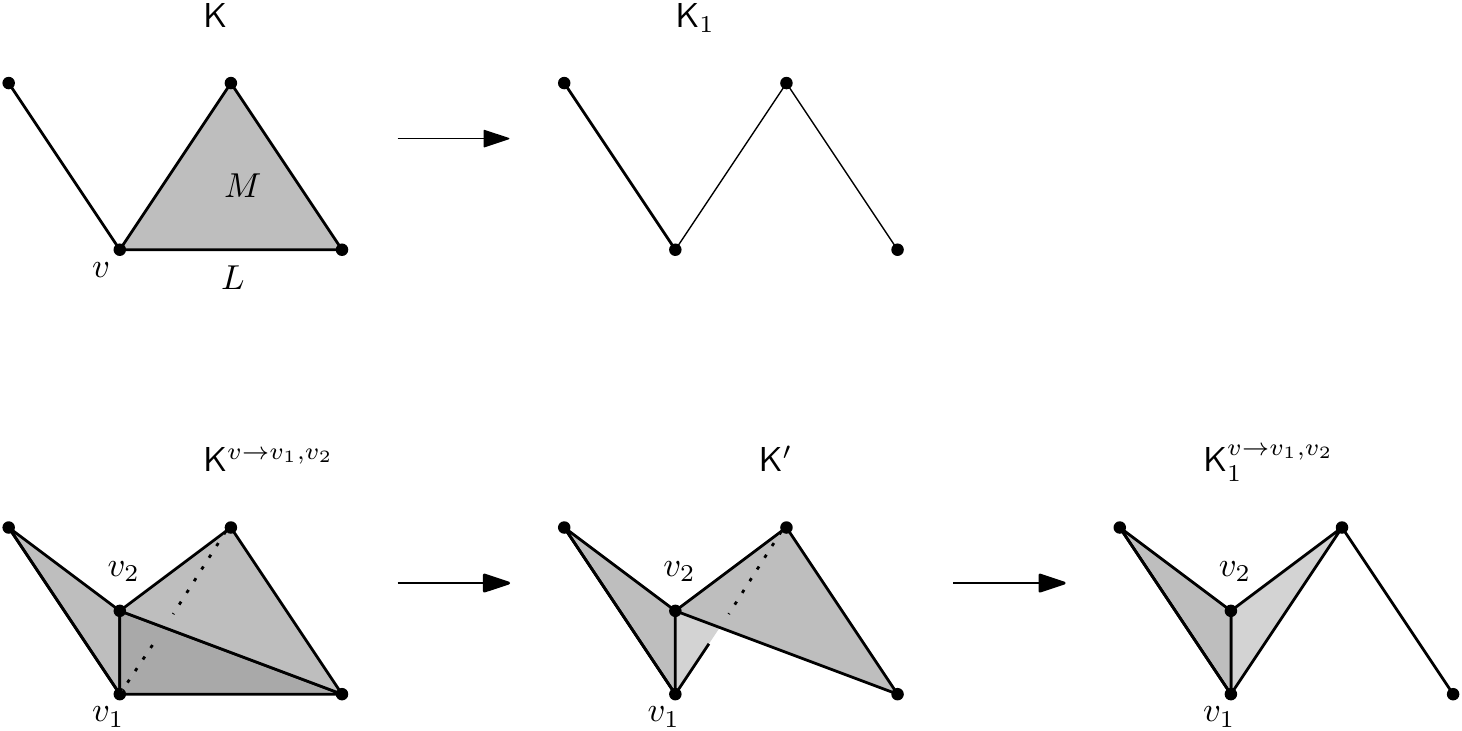}
    \caption{Collapses from $\KK^{v \to v_1, v_2}$ to $\KK_1^{v \to v_1, v_2}$
    if $v \in L$.}
    \label{f:doubled_in_L}
  \end{center}
\end{figure}
\end{proof}

\iffull \paragraph{Density of $P_{\kvec}(\nvec, d, \rvec)$.}
\fi
Now, we will provide a formula for the density of $P_{\kvec}(\nvec, d, \rvec)$.
In the following computations we also set $\delta_\kvec(\nvec, d,
  \rvec)= p_{\kvec}(\nvec,d,\rvec)/\big|\binom N \kvec \big|$ using the notation from
    the definition of $P_{\kvec}(\nvec, d, \rvec)$. We get

\begin{align*}
  p_{\kvec}(\nvec,d,\rvec) &= \left | \left \{S\in \binom{N}{\kvec} : |S\cap \bar{R}| \leq d \right \} \right | \\
  &=\sum_{\lvec=(\ell_1, \dots, \ell_{c}) \in L_{\kvec}(d)}\left | \left \{S\in \binom{N}{\kvec} : |S_i\cap \bar{R}_i| = l_i\right \} \right | \\
  &=\sum_{\lvec=(\ell_1, \dots, \ell_{c}) \in L_{\kvec}(d)} \prod_{i=1}^{c} \binom{n_i-r_i}{l_i}\binom{r_i}{k_i-l_i}.
\end{align*}

Then, using $(x)_m := x\cdot(x-1)\cdots(x-(m-1))$, the density is given by
\begin{equation}
  \label{e:pk-density}
  \delta_{\kvec}(\nvec,d,\rvec)= 
  \frac{p_{\kvec}(\nvec,d,\rvec)}{\prod_{i=1}^{c}\binom{n_i}{k_i}}
  =
  \frac{\sum\limits_{\lvec=(\ell_1, \dots, \ell_{c})
  \in L_{\kvec}(d)} \prod\limits_{i=1}^{c} \binom{k_i}{\ell_i} (n_i-
  r_i)_{\ell_i}(r_i)_{k_i-\ell_i} }{\prod_{i=1}^{c}(n_i)_{k_i}}. 
\end{equation}

\begin{proof}[Proof of Proposition~\ref{p:kcfh_optimal_collapsible}]
  For contradiction, let us assume that for every $i\in [c]$ we have that
  $\dim(\KK[V_i]) \leq \beta_in_i -1$. Let us set $r_i := \dim(\KK[V_i]) + 1 \leq
  \beta_in_i$. 
  Note that the conclusion of
  Theorem~\ref{t:c-k-p} can be restated as $\delta_\kvec(\KK) \leq
  \delta_\kvec(\nvec, d,
    \rvec)$.

  Now we get
  \begin{align*}
    \delta_\kvec(\KK) &\leq \liminf_{m\to \infty} \delta_\kvec(\KK_{\langle 
    m\rangle}) \hfill \text{ by Lemma~\ref{l:boosting}(i)}\\
    &\leq \liminf_{m\rightarrow \infty}\delta_\kvec{(m\nvec,d,m\rvec)} \text{ by Theorem~\ref{t:c-k-p} using
    Lemma~\ref{l:boosting}(ii)}\\
    &\leq \liminf_{m\rightarrow \infty}\delta_\kvec{(m\nvec,d,\lfloor m n_i
    \beta_i\rfloor)} \text{ using $r_i \leq \beta_i n_i$ and monotonicity of
    $p_\kvec(\nvec, d, \rvec)$ in $\rvec$}\\
    &= \liminf_{m\rightarrow \infty}\frac{\sum\limits_{\lvec=(\ell_1, \dots,
    \ell_{c})\in L_{\kvec}(d)} \prod\limits_{i=1}^{c} \binom{k_i}{\ell_i} (mn_i- \lfloor m n_i \beta_i
    \rfloor)_{\ell_i}(\lfloor m n_i \beta_i \rfloor)_{k_i-\ell_i}
    }{\prod_{i=1}^{c}(mn_i)_{k_i}} \text{ by~\eqref{e:pk-density}} \\
    &= \sum_{\lvec=(\ell_1, \dots, \ell_{c})\in L_{\kvec}(d)}
    \prod_{i=1}^{c} \binom{k_i}{\ell_i}
    (1-\beta_i)^{\ell_i}(\beta_i)^{k_i-\ell_i}\\
    &= \alpha_\kvec(d, \bbeta)
  \end{align*}
  which is a contradiction with the assumptions.
\end{proof}

\begin{remark}
It would be much more natural to try to avoid boosting the complex and show directly
  $\delta_k(\KK) \leq
  \delta_k(\nvec, d,
  \rvec) \leq \alpha_{\kvec}(d, \bbeta)$ in the proof of
  Proposition~\ref{p:kcfh_optimal_collapsible}. The former inequality follows
  from Theorem~\ref{t:c-k-p}. However, the latter inequality turned out to be
  somewhat problematic for us when we attempted to show it directly from the
  definition of $\alpha_\kvec(d, \bbeta)$ and from~\eqref{e:pk-density}. Thus,
  in our computations, we take an advantage of the fact that the computations
  in the limit are easier.
\end{remark}



\iffull \paragraph{Tightness of Theorem~\ref{t:kcfh_optimal_collapsible}.} 
\fi
We conclude
this section by showing that the bound given in
Theorem~\ref{t:kcfh_optimal_collapsible} is tight. 

Let us fix $c, d \in \N$, $\kvec = (k_1, \dots, k_{c}) \in \N_0^{c}$ with $k
:= k_1 + \cdots + k_{c} \geq d+1$ and $\bbeta = (\beta_1, \dots,
\beta_{c}) \in (0, 1]^{c}$ as in the statement of
Theorem~\ref{t:kcfh_optimal_collapsible}. Let $0 \leq \alpha' <
\alpha_{\kvec}(d, \bbeta)$. We will find a complex $\KK$ which contains at
least $\alpha'|\binom{N}{\kvec}|$ $\kvec$-colorful faces while $\dim \KK[N_i] <
\beta_i n_i - 1$ for every $i \in [c]$ (using the notation from the statement
of Theorem~\ref{t:kcfh_optimal_collapsible}). 

Similarly as in the proof of Theorem~\ref{t:kcfh_optimal_collapsible} let us consider $\varepsilon > 0$ such that $\bbeta - \eepsilon \in (0,
  1]^{c}$ for $\eepsilon = (\varepsilon, \dots, \varepsilon) \in
  (0,1]^{c}$. In addition, because $\alpha_{\kvec}(d, \bbeta)$ is continuous
  in $\bbeta$ due to its definition~\eqref{e:alpha_kdbeta}, we may pick
  $\varepsilon$ such that $\alpha' < \alpha_{\kvec}(d, \bbeta - \eepsilon)$.
  For simplicity of notation, let $\bbeta' = (\beta'_1, \dots, \beta'_{c}) :=
  \bbeta - \eepsilon$.

Now we pick a positive integer $m$ and set $\nvec = (m, \dots, m) \in
\N^{c}$, that is, $n_1 = \cdots = n_{c} = m$ and $n = cm$ in our standard
notation. We also set $\rvec = (r_1, \dots, r_{c})$ so that $r_i := \lfloor
\beta'_i m
\rfloor$.\footnote{This choice of $\nvec$ will yield a counterexample where each
color class has equal size. It would be also possible to vary the sizes.}
We assume that $m$ is large enough so that $r_i
\geq k_i$ for each $i \in [c]$. We 
define families $N_i$ of convex sets in $\R^d$ so that each $N_i$ contains
$r_i$ copies of $\R^d$ and $m - r_i$ hyperplanes in general position. We also
assume that the collection of all hyperplanes in $N_1, \dots, N_{c}$ is in
general position. We set $\KK$ to be the nerve of the family $N = N_1 \sqcup
\cdots \sqcup N_{c}$. In particular $\KK$ is $d$-representable (therefore
$d$-collapsible as well).

First, we check that $\dim \KK[N_i] <\beta_i m - 1$ provided that $m$ is large enough.
A subfamily of $N_i$ with nonempty intersection contains at most $d$
hyperplanes from $N_i$. Therefore $\dim \KK[N_i] < r_i + d = \lfloor \beta'_i m
\rfloor + d < \beta_i m -1$ for $m$ large enough.

Next we check that $\KK$ contains at
least $\alpha'|\binom{N}{\kvec}|$ $\kvec$-colorful faces provided that $m$ is
large enough. Partitioning $N_i$ so that $R_i$ is the subfamily of the copies
of $\R^d$ and $\bar R_i$ is the subfamily of hyperplanes, we get 
\[f_{\kvec}(K) = p_{\kvec}(\nvec, d, \rvec)
\]
from the definition of $p_{\kvec}(\nvec, d, \rvec)$.
Therefore~\eqref{e:pk-density} gives
\[
  \delta_{\kvec}(\KK)= 
  \frac{\sum\limits_{\lvec=(\ell_1, \dots, \ell_{c})
  \in L_{\kvec}(d)} \prod\limits_{i=1}^{c} \binom{k_i}{\ell_i} (m-
  \lfloor \beta'_i m\rfloor)_{\ell_i}(\lfloor \beta'_i m \rfloor)_{k_i-\ell_i}
  }{\prod_{i=1}^{c}(m)_{k_i}}. 
\]
Passing to the limit (considering the dependency of $\KK$ on $m$), we get
\[
  \lim_{m \to \infty} \delta_{\kvec}(\KK) = \sum\limits_{\lvec=(\ell_1, \dots,
  \ell_{c})
    \in L_{\kvec}(d)} \prod\limits_{i=1}^{c} \binom{k_i}{\ell_i} (1 -
    \beta'_i)^{\ell_i}(\beta'_i)^{k_i - \ell_i} = \alpha_{\kvec}(d, \bbeta').
\]
Therefore, for $m$ large enough $\KK$ contains at
least $\alpha'|\binom{N}{\kvec}|$ $\kvec$-colorful as $\alpha' <
\alpha_{\kvec}(d, \bbeta')$.

\section{A topological version?}
\iffull
A simplicial complex $\KK$ is \emph{$d$-Leray} if the $i$th reduced
homology group $\tilde H_i(\LL)$ (over $\Q$) vanishes for every induced subcomplex
$\LL \leq \KK$ and every $i \geq d$. As we already know, every
$d$-representable complex is $d$-collapsible, and in addition every
$d$-collapsible complex is $d$-Leray~\cite{wegner75}. Helly-type theorems
usually extend to $d$-Leray complexes and such extensions are interesting
because they allow topological versions of Helly-type when collections of
convex sets are replaced with good covers. We refer to several concrete
examples~\cite{helly30, kalai2005topological, alon-kalai-matousek-meshulam02}
or to the survey~\cite{tancer13survey}. 

We believe that it is possible to extend
Theorem~\ref{t:cfh_optimal_collapsible} to $d$-Leray complexes:
\fi 

\begin{conjecture}[The optimal colorful fractional Helly theorem for
    $d$-Leray complexes]
 \label{c:frac-colorful-helly-leray}
 Let $\KK$ be a $d$-Leray simplicial complex with the set of vertices $N
  = N_1
  \sqcup \cdots \sqcup N_{d+1}$ divided into $d+1$ disjoint subsets. Let $n_i
  := |N_i|$ for $i \in [d+1]$ and assume that $\KK$ contains
  at least $\alpha n_1 \cdots n_{d+1}$ colorful $d$-faces for some $\alpha \in
  (0, 1]$. Then there is $i \in [d+1]$ such that $\dim \KK[N_i] \geq (1 -
  (1-\alpha)^{1/(d+1)}) n_i - 1$. 
\end{conjecture}

In fact, our original approach how to prove
Theorem~\ref{t:cfh_optimal_collapsible} was to prove directly
Conjecture~\ref{c:frac-colorful-helly-leray}. Indications that this could be
possible are that both the optimal fractional Helly
theorem~\cite{alon-kalai-matousek-meshulam02, kalai02} and the colorful Helly
theorem~\cite{kalai2005topological} hold for $d$-Leray complexes. In addition,
there is a powerful tool, algebraic shifting, developed by
Kalai~\cite{kalai02}, which turned out to be very useful in attacking similar
problems.

In the remainder of this section we briefly survey a possible approach towards
Conjecture~\ref{c:frac-colorful-helly-leray} but also the difficulty that we
encountered. Because we do not really prove any new result in this section, our
description is only sketchy.

Our starting point is the proof of the optimal fractional Helly theorem
for $d$-Leray complexes. The key ingredient is the following theorem of
Kalai~\cite[Theorem~13]{alon-kalai-matousek-meshulam02}.

\begin{theorem}
\label{t:optimal-fractional-Helly}
  Let $\KK$ be a $d$-Leray complex and $f_0(\KK) = n$. Then $f_d(\KK) >
  \binom{n}{d+1} - \binom{n-r}{d+1}$ implies $f_{d+r}(\KK) > 0$ (where $f(\KK)$
  denotes the $f$-vector of $\KK$).
\end{theorem}

As far as we can judge, the only proof of
Theorem~\ref{t:optimal-fractional-Helly} in the literature 
follows from the first and the third sentence in the following remark in~\cite{kalai02}:
\medskip

``It is not hard to see (although it has been overlooked for a long time) that
the class of $d$-Leray complexes (for some $d$) with complete
$(d-1)$-dimensional skeletons is precisely the Alexander dual of the class of
Cohen-Macaulay complexes. This observation implies that the fact that shifting
preserves the Leray property easily follows from the fact that shifting
preserves the Cohen-Macaulay property. Moreover, it shows that the
characterization of face numbers of $d$-Leray complexes follows from the
corresponding characterization for Cohen-Macaulay complexes.'' 
\medskip

For completeness we add that the characterization of face numbers of Cohen-Macaulay
complexes has been done by Stanley~\cite{stanley75}. Given a simplicial complex
$\KK$ on vertex set $V$ its \emph{Alexander dual} is a simplicial complex
defined as $\KK^* := \{\sigma \subseteq V\colon V \setminus \sigma \not\in
\KK\}$. We skip the definition of Cohen-Macaulay complex because we will only
use it implicitly but we refer, for example, to~\cite[\S4]{kalai02} for more
details.
%

A simplicial complex $\KK$ on vertex set $[n]$ is called \emph{shifted} if for all integers $i$ and $j$ with $1 \leq i < j \leq n$ and all
faces $A$ of $\KK$ such that $j \in A$ and $i \notin A$, the set $(A \setminus \{j\}) \cup \{i\}$ is a face of $\KK$.
\emph{Exterior algebraic shifting} is a function that associates to a
simplicial complex $\KK$ a shifted complex $\KK^e$ , while preserving many
interesting invariants of $\KK$. Below we list some properties of exterior
algebraic shifting that we will use. A simplicial complex is \emph{pure}
if all its inclusion-maximal faces have the same dimension.
\begin{theorem}
  \label{t:shifting}
  \begin{enumerate}[(i)]
  \item ~\cite{kalai02}[Theorem 2.1] Exterior algebraic shifting preserves the $f$-vector.  
  \item ~\cite{kalai02}[Theorem 4.1] If $\KK$ is Cohen-Macaulay, then $\KK^e$ is
      Cohen-Macaulay, in particular, pure.
  \item ~\cite{kalai02}[3.5.6] Exterior algebraic shifting and Alexander duality commute.
  \end{enumerate}
 \end{theorem}

The next lemma is a possible replacement of the third sentence in Kalai's
remark how to prove Theorem~\ref{t:optimal-fractional-Helly}. We prove it as
motivation for the tools we would need in the colorful scenario.

\begin{lemma}
  \label{l:extremal}
  Let  $\KK$ be a $d$-Leray complex on $[n]$ with complete $(d-1)$-skeleton and
  let $D=\dim(\KK) + 1$. Then $\KK^e \subseteq
  \Delta_{D-d-1}*\Delta^{(d-1)}_{n-D+d-1}$.
\end{lemma}
\begin{proof}
  By the first sentence of Kalai's remark, the Alexander dual $\KK^*$ of
  $\KK$ is a Cohen-Macaulay complex. By the definition of Alexander dual,
  it has dimension $n-d-2$ and contains complete
$(n-D-2)$-skeleton. Hence, properties (i) and (ii) of Theorem~\ref{t:shifting}
imply that the exterior algebraic shifting $(\KK^*)^e$ of $\KK^*$ is a pure
shifted complex of dimension $n-d-2$ with complete $(n-D-2)$-skeleton. If we
take any subset $A$ of size $n-D-1$ in $(\KK^*)^e$, then $A$ is a face and by
purity there must be a face of size $n-d-1$ that contains $A$. Now, since
$(\KK^*)^e$ is shifted we have that $\{1,2,\dots,D-d \} \cup A\in(\KK^*)^e$.
  This implies that $\Delta_{D-d-1}*\Delta^{(n-D-2)}_{n-D+d-1} \subseteq (\KK^*)^e =
  (\KK^e)^*$, by Theorem~\ref{t:shifting}(iii). Taking Alexander dual from both sides proves the first part of the
  statement. (Using that $(\LL^*)^* = \LL$; $\LL_1 \subseteq \LL_2
  \Rightarrow \LL_2^* \subseteq \LL_1^*$; and
  $(\Delta_{D-d-1}*\Delta^{(n-D-2)}_{n-D+d-1})^* =
  \Delta_{D-d-1}*\Delta^{(d-1)}_{n-D+d-1}$.)
\end{proof}

For completeness, Theorem~\ref{t:optimal-fractional-Helly} quickly follows from
Lemma~\ref{l:extremal}. Indeed, if $\KK$ is $d$-Leray such that $f_{d+r}(\KK) =
0$, then $D := \dim \KK + 1 \leq d + r$. In addition, we can assume without
loss of generality that $\KK$ contains complete $(d-1)$-skeleton. Consequently,
Lemma~\ref{l:extremal} gives $f_d(\KK) = f_d(\KK^e) \leq
f_d(\Delta_{D-d-1}*\Delta^{(d-1)}_{n-D+d-1}) = \binom{n}{d+1} -
\binom{n-D+d}{d+1} \leq \binom{n}{d+1} -
\binom{n-r}{d+1}$.


Now, in order to attack Conjecture~\ref{c:frac-colorful-helly-leray}, we
would like to do something similar in colorful setting. In particular, we need
to preserve the colorful $f$-vector.
Babson and Novik~\blue{\cite{babson2004}} give a definition of colorful algebraic shifting which
preserves the colorful $f$-vector. Nevertheless, the conjecture does not follow
immediately from their result as the Alexander dual of a $d$-Leray complex is
not in general balanced.

\iffull
\section*{Acknowledgments}
We thank Xavier Goaoc and Eran Nevo for providing us with useful references.

\bibliographystyle{alpha}
\bibliography{cfh}
\fi

\end{document}